\documentclass{article}
\usepackage[utf8]{inputenc}
\usepackage{amssymb,amsfonts,amsmath,amsxtra,amsthm}
\usepackage{graphicx}
\usepackage{cite}
\usepackage{hyperref}
\usepackage{enumerate}

\def\de{\delta}
\def\e{\varepsilon}
\def\o{\omega}
\def\O{\Omega}
\def\RR{{\mathord{I\!\! R}}}
\def\NN{{\mathbb{N}}}
\def\div{{\mathrm{div}}}

\def\Ldeux{{L^2}}
\def\Linf{{L^{\infty}}}
\def\Vert{||}

\usepackage{color}

\definecolor{gr}{rgb}   {0.,   0.69,   0.23 }
\definecolor{bl}{rgb}   {0.,   0.5,   1. }
\definecolor{mg}{rgb}   {0.85,  0.,    0.85}
\definecolor{yl}{rgb}   {0.8,  0.7,   0.}

\def\dd#1#2{\dfrac{\partial #1}{\partial #2}}
\def\div{{\mathrm{div}}}
\def\curl{{\mathrm{curl}}}
\def\LL{{\mathbf{L}}}
\def\HH{{\mathbf{H}}}

\def\integre#1#2#3#4{\int\limits_{ #1 }^{ #2 }{ #3 \: \mathrm{d} \mathrm{#4}}}
\DeclareRobustCommand{\cvfort}{\relbar\joinrel\rightarrow}
\newcommand{\cvfaible}{}
\DeclareRobustCommand{\cvfaible}{\relbar\joinrel\rightharpoonup}

\newtheorem{prpstn}{Proposition}[section]
\newtheorem{lmm}{Lemma}[section]
\newtheorem{thrm}{Theorem}[section]
\newtheorem{dfntn}{Definition}[section]

\title{Convergence of a Vector Penalty Projection Scheme for the Navier-Stokes Equations with moving body}
\author{
Vincent Bruneau$^1$, Adrien Doradoux$^{1,2}$, Pierre Fabrie$^2$ \\
\small $^1$ Universit\'e de Bordeaux, IMB, CNRS UMR5251, 351 cours de la lib\'eration, 33405 Talence - France \\
\small $^2$ Bordeaux INP, Institut de Math\'ematiques de Bordeaux, CNRS UMR5251, ENSEIRB-MATMECA, Talence - France
}
\date{\today}

\begin{document}
\maketitle
\begin{abstract}
In this paper, we analyse a {\it Vector Penalty Projection Scheme} (see \cite{AN08}) to treat the displacement of a moving body in incompressible viscous flows in the case where the interaction of the fluid on the body can be neglected. The presence of the obstacle inside the computational domain is treated with a penalization method introducing a parameter $\eta$.\\
We show the stability of the scheme and that the pressure and velocity converge towards a limit when the penalty parameter $\e$, which induces a small divergence and the time step $\delta t$ tend to zero with a proportionality constraint $\e = \lambda \delta t$.\\
Finally, when $\eta$ goes to $0$, we show that the problem admits a weak limit which is a weak solution of the Navier-Stokes equations with no-sleep condition on the solid boundary.
 \end{abstract}
\vspace{5pt}
\begin{center}
\small \textbf{Résumé}
\end{center}
\vspace{-5pt}
Dans ce travail nous analysons un schéma de projection vectorielle (voir \cite{AN08}) pour traiter le d\'eplacement d'un corps solide dans un fluide visqueux incompressible dans le cas o\`u l'interaction du fluide sur le solide est négligeable. La pr\'esence de l'obstacle dans le domaine solide est mod\'elis\'ee par une m\'ethode de p\'enalisation.\\
Nous montrons la stabilit\'e du sch\'ema et la convergence des variables vitesse-pression vers une limite quand le param\`etre $\e$ qui assure une faible divergence et le pas de temps $\delta t$ tendent vers $0$ avec une contrainte de proportionalit\'e $\e = \lambda \delta t$. \\
Finalement nous montrons que le probl\`eme converge au sens faible vers une solution des \'equations de Navier-Stokes avec une condition aux limites de non glissement sur la frontière immergée quand le param\`etre de p\'enalisation $\eta$ tend vers $0$.

%

%

\section*{Introduction}
Numerical simulation of unsteady incompressible flows has always been the subject of important investigations. The difficulty arises from the coupling between the velocity and the pressure at each time step due to the incompressibility constraint.\\
Fractional step algorithms are widely used to deal with this problem. Among them, pressure projection methods introduced by Chorin and Temam (\cite{CH68}, \cite{TE69}) reduce the saddle point problem to two distinct elliptic problems on the velocity and the pressure.\\
Recently, {\it Vector Penalty Projection} schemes have been introduced by Angot et al. (\cite{AN08}, \cite{AN12}) and avoid many drawbacks of other projection methods. In particular, the pressure scalar does not need to be computed, which do not impose the resolution of a Poisson type equation that introduce boundary conditions on the pressure. In \cite{AN12c} the authors obtained a second order convergence rate for pressure and velocity in space and time for a second order backward temporal scheme. The convergence towards the Navier-Stokes problem when the penalty parameter (on the divergence) tends to $0$ has been studied in \cite{AN15}.
In these previous works the domain is fixed and no {\it space penalization} is included.  \\
\medskip
Let $\O$ be a simply connected bounded domain of $\RR^d$ ($d=2$ or $3$) and $T>0$.\\
With Dirichlet boundary condition and a source term $f$ given, the penalized Navier Stokes problem reads in presence of a body moving at the velocity $v_s$:
\begin{equation}\label{eq:NavierStokes}
\begin{split}
&\dd{v}{t} + (v . \nabla) v - \div(2 \mu D(v)) + \dfrac{1}{\eta} \chi_{\o(t)}(v-v_s) + \nabla p = f \text{ on } \O \\
& \div(v) = 0 \text{ on }\O \\
& v=0 \text{ on } \partial \O \\
& v(0,x) = v_0(x) \text{ on }\O
\end{split}
\end{equation}
Where $\chi_{\o(t)}$ is the characteristic function of the solid domain $\o(t)$, $v_s$ is the velocity of the moving body and $D$ is the strain rate tensor.\\
\underline{Hypothesis}\\
$\left( \mathcal{H}1 \right)$: We suppose that $v_s$ is the restriction to $\bigcup\limits_{t<T}\o(t)$ of a function $\psi$ such that:
\begin{equation}\label{eq:psi}
\left\lbrace
\begin{split}
& \div(\psi) = 0 \text{ on }\O\\
& \psi \in L^{\infty}(]0,T[;\LL^{\infty}) \cap L^2(]0,T[;\HH^2) \\
& \nabla \psi \in L^{\infty}(]0,T[;\LL^{\infty}) \\
& \dd{\psi}{t} \in L^2(]0,T[;\LL^2)
\end{split}
\right.
\end{equation}
The existence of such function is ensured as the regularity of $v_s$ is sufficient and the moving body does not meet $\partial \O$. \\
$\left( \mathcal{H}2 \right)$: We assume that $f$ belongs to the space $L^2(]0,T[;\LL^2)$.\\
\vspace*{5pt}
The Vector Penalty Projection Scheme is a fractional step method:
\begin{itemize}
\item A predicted velocity $\tilde{v}^{n+1}$ is first computed ($v^n \rightarrow \tilde{v}^{n+1}$):  considering the pressure gradient at the previous time step $t^n$. At the end of this step, the velocity does not respect the free divergence condition.
\item The velocity is then corrected such that $\div(v^{n+1})$ is approximately $0$ at the end of the time step.
\item We finally actualize the pressure gradient $\nabla p^{n+1}$.
\end{itemize}
For all $n \in \NN$ such that $n \delta t \leq T$ the numerical scheme reads:
\begin{equation}\label{eq:vpp_prediction}
\dfrac{\tilde{v}^{n+1}-v^n}{\delta t} + B(v^n,\tilde{v}^{n+1}) - \div( 2 \mu D(\tilde{v}^{n+1})) + \dfrac{1}{\eta} \chi_{\o(t^{n+1})}(\tilde{v}^{n+1}-v_s) + \nabla p^n = f^{n+1}
\end{equation}
\begin{equation}\label{eq:vpp_correction}
\dfrac{\e}{\delta t} \hat{v}^{n+1} - \nabla\,(\div(\hat{v}^{n+1})+\div(\tilde{v}^{n+1})) = 0
\end{equation}
\begin{equation}\label{eq:vpp_gradpression}
\nabla(p^{n+1} - p^n) + \dfrac{1}{\e} \nabla (\div(v^{n+1})) = 0
\end{equation}
Where $v^{n+1} = \tilde{v}^{n+1} + \hat{v}^{n+1}$.\\
It is completed by the following initial and boundary conditions on $\partial \O$:
\begin{equation}\label{eq:vpp_limit}
\tilde{v}^{n+1} = 0 \text{ on } \partial \O, \:\:\: \hat{v}^{n+1}.\nu= 0 \text{ on }\partial \O
\end{equation}
\begin{equation}\label{eq:vpp_initial}
\tilde{v}^0 = v_0 \text{ in } \O, \:\:\: \hat{v}^0 =0 \text{ in }\O
\end{equation}
where $\nu$ is the outward unit normal vector on $\partial \O$.

The space $\O$ is a bounded simply connected domain of $\RR^d$. 
Furthermore, $p^0 \in \LL_0^2$ the $\LL^2$-space of null average functions as well as $\div(v^{n+1})$ since $v^{n+1}.\nu = 0$. Therefore we can show recursively that $p^{n+1}$ has a null average on $\O$, for all $n \in \NN$, solving \eqref{eq:vpp_gradpression} in the space of null average functions.  We finally obtain:
\begin{equation}\label{eq:vpp_pression}
\e (p^{n+1}-p^n) + \div(v^{n+1}) = 0
\end{equation}

\begin{prpstn}[Existence of the iterates]\label{prop:existenceiteres} ~ \\
We suppose that $v^0 \in \HH^1$ and $p^0 \in \LL_0^2$.
Then , for all $n \in \NN$ such that $n \delta t \leq T$, we have:
\begin{equation*}
(v^n,p^n) \in \HH^1 \times \LL_0^2
\end{equation*}
\end{prpstn}
\begin{proof}\label{prop1}
Suppose that $v^n \in \HH^1$ and show $v^{n+1} \in \HH^1$. Equation \eqref{eq:vpp_prediction} is a linear Stokes problem, therefore we easily show that $\tilde{v}^{n+1}$ exists and lies in $\HH_0^1$.\\
We show now the existence of $\hat{v}^{n+1}$. For a simply connected domain, the norm $\Vert \cdot \Vert_{\HH^1}^2$ is equivalent to the norm $\Vert \div(\cdot) \Vert^2_{\LL^2} + \Vert \curl(\cdot) \Vert^2_{\LL^2}$  (see \cite{FO78}). By energy estimates, we obtain that for any $n \in \NN$ there exists a constant $C > 0$ such that:
$$
\dfrac{\e}{\delta t} \Vert \hat{v}^{n+1} \Vert^2_{\LL^2} + \dfrac{1}{2} \Vert \div(\hat{v}^{n+1}) \Vert^2_{\LL^2} \leq C
$$
Moreover, $\hat{v}^{n+1}$ is a gradient, therefore $\curl(\hat{v}^{n+1}) =0$. We deduce that $\hat{v}^{n+1}$ lies in $\HH^1$.\\
Using \eqref{eq:vpp_gradpression}, we deduce that $\nabla p^{n+1}$ lies in $ \HH^{-1}$. As $p^{n+1}$ has a null average on $\O$, we conclude using the Poincar\'e Wirtinger inequality that $p^{n+1} \in \LL_0^2$.
\end{proof}

%
\section{Main result}

In this paper, to each sequence $(v^k)_k$ of functions defined on $\O$ we will associate a sequence of functions $(v_{\delta t})_{\delta t}$ which are the step functions in time $v_{\delta t}$ defined by:
\begin{equation}\label{eq:defconstantparmaille}
v_{\delta t}(t) = v^k \text{ if } t \in [t^k,t^{k+1}[.
\end{equation} 
For the step functions contructed from the sequences of Proposition \ref{prop:existenceiteres} we have the following convergence result.

\begin{thrm}[Convergence when $\e$ and $\delta t$ tend to $0$]\label{th:convergence_epsilon} ~ \\
Let $\O \in \RR^d$ be a simply connected bounded domain and $d=2$ or $3$. Let $\mu >0 $ and $v^0 \in \HH$ where $\HH = \left\lbrace v \in \LL^2,\: \div(v) = 0,\:v.\nu =0 \text{ on }\partial \O \right\rbrace$.\\
We assume that there exists $\lambda > 0$ such that $\e = \lambda \delta t$ and we suppose that hypothesis $\left( \mathcal{H}1 \right)$ and $\left( \mathcal{H}2 \right)$ are verified.\\
Then, up to a subsequence, $(v_{\delta t},p_{\delta t})$ defined by  \eqref{eq:vpp_prediction}-\eqref{eq:vpp_initial} and \eqref{eq:defconstantparmaille} converges towards $(v,p)$, a weak solution of the penalized Navier Stokes problem \eqref{eq:NavierStokes} when $\e = \lambda \delta t$ tends to $0$.
Furthermore, $v$ and $p$ satisfy:
$$
v \in \Linf(]0,T[;\LL^2)\cap \Ldeux(]0,T[;\HH^1)
$$
$$
p \in W^{-1,\infty}(]0,T[;\LL_0^2)
$$
Moreover, this solution is unique for $d=2$.
\end{thrm}
The proof of this theorem is given in section \ref{sec:convergence_epsilon}.
We complete this last result by giving a new proof of existence of weak solutions to the Navier Stokes equation in presence of a moving body.\\
To do so, we focus on the behavior of the solution when $\eta$ tends to $0$ and denote $v_{\eta}$ the weak limit obtained in Theorem \ref{th:convergence_epsilon}.
\begin{thrm}[Convergence when $\eta$ tends to $0$]\label{th:convergence_eta} ~ \\
When $\eta$ tends to $0$, the sequence $\left( v_{\eta} \right)_{\eta}$ weakly converges towards a limit $v$ which satisfies
\begin{equation*}
v_{|\partial \o(t)} = v_s(t)
\end{equation*}
Furthermore, $v_{|\O \setminus \o(t)}$ is a weak solution of the Navier Stokes equations on $\O \setminus \o(t)$.\\
Moreover, there exists $h \in \mathbf{W}'$ such that:
\begin{equation}
\dfrac{1}{\eta} \chi_{\o(t)}(v-v_s) \cvfaible h \text{ in } \mathbf{W}'
\end{equation}
Where $\mathbf{W}$ is defined in section \ref{sec:notations}.
\end{thrm}
The proof of this result is given in Section \ref{sec:convergence_eta}.\\
Before proving these theorems, we first prove the stability of the numerical scheme.\\
In order to perform energy estimates, we need to build a lifting of the velocity. In the following to the function $v$ (resp.  ${v}^{n}$, $\tilde{v}^{n}$, $\hat{v}^{n}$,  ${v}_{\delta t}$) we will associate the lifting functions $w$ (resp.  ${w}^{n}$, $\tilde{w}^{n}$, $\hat{w}^{n}$,  ${w}_{\delta t}$) defined by substraction of the function $\psi$:
\begin{equation}\label{eq:def_lifting}
w=v-\psi, \quad {w}^{n}={v}^{n}-\psi , \quad  \tilde{w}^{n}= \tilde{v}^{n}-\psi , \quad \hat{w}^{n}= \hat{v}^{n}, \quad  {w}_{\delta t}={v}_{\delta t} -\psi .
\end{equation}
The system \eqref{eq:vpp_prediction}-\eqref{eq:vpp_initial} becomes:
\begin{equation}\label{eq:vpp_lifting_prediction}
\begin{split}
\dfrac{\tilde{w}^{n+1}-w^n}{\delta t} &+ B(w^n,\tilde{w}^{n+1})+B(\psi^n,\tilde{w}^{n+1})\\
& - \div(2 \mu D(\tilde{w}^{n+1})) + \dfrac{1}{\eta} \chi_{\o(t^{n+1})} \tilde{w}^{n+1} + \nabla p^n = F^{n+1} - B(w^n,\psi^{n+1})
\end{split}
\end{equation}
\begin{equation}\label{eq:vpp_lifting_correction}
\dfrac{\hat{w}^{n+1}}{\delta t} - \dfrac{1}{\e} \nabla (\div(\tilde{w}^{n+1} + \hat{w}^{n+1})) = 0
\end{equation}
\begin{equation}\label{eq:vpp_lifting_gradpression}
\nabla(p^{n+1}-p^n) + \dfrac{1}{\e} \nabla (\div(w^{n+1}))=0
\end{equation}
\begin{equation}\label{eq:vpp_lifting_initial}
w^0(x) = v^0 - \psi^0
\end{equation}
Where $F^{n+1} = f^{n+1}- \dfrac{\psi^{n+1}-\psi^n}{\delta t} + \div(2 \mu D(\psi^{n+1})) - B(\psi^n,\psi^{n+1})$.\\
\begin{prpstn}[Stability]\label{prop:stability} ~ \\
For $\mu>0$, $v^0 \in \LL^2$, $p^0 \in \LL_0^2$, $\nabla p^0 \in \LL^2$ and $f \in L^2(]0,T[;\LL^2)$ given, if we assume that $\psi$ verifies the hypothesis $\left( \mathcal{H}1\right)$.\\
Then, for any $T>0$ there exists $C > 0$ such that
 \begin{enumerate}[i)]
\item $w_{\de t}$ is bounded in $\Linf(]0,T[;\LL^2)$.
\item  $\tilde{w}_{\de t}(\cdot + \delta t) - w_{\de t}(\cdot)$ is bounded in $L^{\infty}(]0,T-\delta t[; \LL^2)$ and $$\Vert \tilde{w_{\de t}}(\cdot + \delta t) - w_{\de t}(\cdot) \Vert_{L^{\infty}(]0,T-\delta t[;\LL^2)} \leq C \sqrt{\de t}$$
\item  $\nabla \tilde{w}_{\de t}$ is bounded in $L^2(]0,T[;\LL^2)$.
 \end{enumerate}
\end{prpstn}
The proof of this proposition can be found in section \ref{sec:stability}.

%
\section{Mathematical recalls}
\subsection{Notations}\label{sec:notations}
Let us define the main notations of this paper.
\begin{itemize}
\item For $p>0$, $\LL^p = L^p(\O)$.
\item For $p>0$, $\LL_0^p = \left\lbrace v \in \LL^p ;\, \integre{\O}{}{v}{x} =0 \right\rbrace $.
\item For $p \in \RR$, $\HH^p = H^p(\O)$, the classical Sobolev space.
\item For $p>0$, $\HH_0^p  = \left\lbrace v \in \HH^p;\, v_{| \partial \O} = 0 \right\rbrace$.
\item $\HH = \left\lbrace v \in \LL^2 ;\, \div(v) =0 \text{ on } \O;\, \gamma_{\nu}(v) = 0 \right\rbrace$.
\item $\mathbf{V} = \left\lbrace v \in \HH_0^1 ;\, \div(v) =0 \text{ on }\O \right\rbrace$.
\item $\mathbf{W} = \left\lbrace v \in L^2(]0,T[;\mathbf{V});\, \dfrac{\partial v}{\partial t} \in L^2(]0,T[;\mathbf{V}');\,v(T)=0 \right\rbrace$.
\end{itemize}
Where $\gamma_{\nu}$ is the trace operator. \\
We recall the definition of the Nikolskii spaces.
\begin{dfntn}\label{def:nikolskii}
Let $E$ be a Banach space and $f \in L^1(]0,T[;E)$.\\
For $1 \leq q < \infty$ and $0 < \sigma < 1$, we define the Nikolskii space $N_q^{\sigma}(]0,T[;E)$ by
$$
N_q^{\sigma} = \left\lbrace f \in L^q(]0,T[;E),\: \sup\limits_{0<h<T} \dfrac{\Vert \tau_h f -f \Vert_{L^q(]0,T-h[;E)}}{h^{\sigma}} < \infty \right\rbrace
$$
and for $f \in N_q^{\sigma}(]0,T[;E)$ we define the associated norm
\begin{equation*}
\Vert f \Vert_{N_q^{\sigma}(]0,T[;E)} = \left\lbrace \Vert f \Vert^q_{L^q(]0,T[;E)} + \sup\limits_{0<h<T} \left( \dfrac{1}{h^{\sigma}} \Vert \tau_h f - f \Vert_{L^q(]0,T-h[;E)} \right)^q \right\rbrace^{\frac 1 q}
\end{equation*}
Where $\tau_h$ is the translation operator defined by $\tau_h f (\cdot) = f(\cdot+h)$
\end{dfntn}
We denote $(\, \cdot\, ,\,\cdot \,)_{\LL^2}$ the usual scalar product on $\LL^2$ and $<\, \cdot\, ,\,\cdot \,>_{E',E}$ the duality bracket.\\
In all the paper the constants are always denoted $C$. 
\subsection{Mathematical properties}
To deal with the nonlinear convective term, we use the bilinear form $B$ introduced by Temam (see \cite{TE68} and \cite{TE84}).\\
For $u \in \HH^1$ and $v \in \HH_0^1$,
\begin{equation}\label{eq:formeB}
B(u,v) = (u . \nabla) v + \dfrac{1}{2} \div(u) v
\end{equation}
Taking the scalar product of $B(u,v)$ by $w \in \HH_0^1$ and integrating by part the second term, we obtain the associated trilinear form $b$: 
\begin{equation}\label{eq:formetrib}
b(u,v,w) = \dfrac{1}{2}\integre{\O}{}{(u . \nabla)v.w}{x}  - \dfrac{1}{2} \integre{\O}{}{(u . \nabla) w . v}{x} 
\end{equation}
The trilinear form $b$ satisfies the antisymmetry property $b(u,v,w) =- b(u,w,v)$ and $b(u,v,v) =0$. 
We now recall the discrete Gronwall Lemma (see \cite{HO09}, \cite{GI07}, \cite{QU14}).
\begin{lmm}[Discrete Gronwall lemma \cite{HO09}]\label{lem:gronwall_discrete} ~ \\
Let $(y_n)$, $(f_n)$ and $(g_n)$ three non-negative sequences such that:
$$
y_n \leq f_n + \sum\limits_{k=0}^{n-1}{g_k \, y_k} \text{ for } n \geq 0
$$
Then,
$$
y_n \leq f_n + \sum\limits_{k=0}^{n-1}{f_k g_k \exp \left( \sum\limits_{ j =k+1}^ {n-1}{g_j}\right)} \text{ for } n \geq 0
$$
\end{lmm}

In order to prove the convergence of the velocity we will need the following analysis result (\cite{BO13}, p99).\\
Let $X$ and $Y$ two Banach spaces such that $X$ is embedded in a continuous and dense way into $Y$, and let $T>0$ and $p,q$ satisfy $1 \leq p,q \leq + \infty$. We denote:
$$
E_{p,q} = \left\lbrace
u \in L^p(]0,T[,X), \: \dfrac{d u}{d t} \in L^q(]0,T[,Y)
 \right\rbrace
$$

\begin{prpstn}\label{prp:initialdata}
Any element $u$ of $E_{p,q}$ (defined almost everywhere) possesses a continuous representation on  $[0,T]$ with values in $Y$, and the embedding of $E_{p,q}$ into $\mathcal{C}^0([0,T],Y)$ is continuous.\\
Moreover, for all $t_1,t_2 \in [0,T]$ we have
$$
u(t_2)-u(t_1) = \integre{t_1}{t_2}{\dfrac{d u}{d t}}{t}
$$
where it is understood that we have identified $u$ and its continuous representation.
\end{prpstn}

Finally, let us formulate an important compactness theorem, which will be useful to prove the strong convergence of the velocity in section \ref{sec:convergence_epsilon}.\\
\begin{thrm}[Simon \cite{SI87}]\label{th:simon} ~ \\
Let $B_0$, $B_1$ and $B_2$ three Banach spaces with $B_0 \subset B_1 \subset B_2$. We suppose that the embedding of $B_0$ in $B_1$ is compact and the embedding of $B_1$ in $B_2$ is continuous.\\
Then, for all $1 \leq q \leq + \infty$ and $0 < \sigma < 1$, the embedding
\begin{equation*}
L^q(]0,T[;B_0) \cap N_q^{\sigma}(]0,T[;B_2) \hookrightarrow L^q(]0,T[;B_1)
\end{equation*} 
is compact.
\end{thrm}
%
\section{Stability analysis}
\label{sec:stability}
In this section we prove the stability of the numerical scheme on the problem \eqref{eq:vpp_lifting_prediction}-\eqref{eq:vpp_lifting_initial}. To do so, energy estimates are performed on the prediction, the correction  and pressure  equations respectively given  by \eqref{eq:vpp_lifting_prediction}, \eqref{eq:vpp_lifting_correction} and \eqref{eq:vpp_lifting_gradpression}. 
Then we will obtain an upper bound on  $\div(w) $ in $L^2(]0,T[;\LL^2)$. 
By definition of $w_{\de t}$ (see \eqref{eq:defconstantparmaille} and \eqref{eq:def_lifting}) Proposition \ref{prop:stability} will be a direct consequence of the following result
\begin{prpstn}[Stability]\label{prop:stability2} ~ \\
For $\mu>0$, $v^0 \in \LL^2$, $p^0 \in \LL_0^2$, $\nabla p^0 \in \LL^2$ and $f \in L^2(]0,T[;\LL^2)$ given, if we assume that $\psi$ verifies the hypothesis $\left( \mathcal{H}1\right)$.
Then, for any $T>0$ there exists $C > 0$ such that
 \begin{enumerate}[i)]
\item the sequence $(w^k)_k$ is bounded in $\LL^2(\O)$.
\item the sequence  $\Big(\frac{\tilde{w}^{k+1} - w^{k}}{\sqrt{\de t}}\Big)_k$ is bounded in $\LL^2(\O)$. 
\item $\sum\limits_{k=0}^{n-1}{ \delta t \Vert \nabla \tilde{w}^{k+1} \Vert^2_{\LL^2}}$ is uniformly bounded with respect to $\de t$.
 \end{enumerate}
\end{prpstn}
\medskip
\begin{proof}[Proof of Proposition \ref{prop:stability2}] ~ \\
We obtain the result using several energy estimates as in \cite{AN15} for homogeneous Navier-Stokes flows. In our estimates, an additional term appears due to the penalization term on the moving body.\\
Taking $\tilde{w}^{n+1}$ as a test function in \eqref{eq:vpp_lifting_prediction}, we obtain:
\begin{equation*}
\begin{split}
&\dfrac{1}{\delta t} \left( \tilde{w}^{n+1}-w^n,\tilde{w}^{n+1} \right)_{\LL^2}
+ 2 \mu \left( D(\tilde{w}^{n+1}),D(\tilde{w}^{n+1})\right)_{\LL^2}
+ \dfrac{1}{\eta} \integre{\O}{}{\chi_{\o(t^{n+1})} |\tilde{w}^{n+1}|^2}{x}\\
& + b(w^n,\tilde{w}^{n+1},\tilde{w}^{n+1}) + b(\psi^n,\tilde{w}^{n+1},\tilde{w}^{n+1}) 
+ \left( \nabla p^n, \tilde{w}^{n+1} \right)_{\LL^2}
\\
& 
= \left( F^{n+1}, \tilde{w}^{n+1} \right)_{\LL^2} -  b(w^n,\psi^{n+1},\tilde{w}^{n+1})
\end{split}
\end{equation*}
The diffusion term is integrated by parts, the Korn inequality (\cite{KO06}, \cite{KO09}) is then used to obtain the lower bound:
$$
\Vert \nabla \tilde{w}^{n+1} \Vert^2_{\LL^2} \leq 2 \Vert D(\tilde{w}^{n+1}) \Vert^2_{\LL^2}.
$$
The convective terms $b(w^n,\tilde{w}^{n+1},\tilde{w}^{n+1})$ and $b(\psi^n,\tilde{w}^{n+1},\tilde{w}^{n+1})$ vanish by antisymmetry of the trilinear form $b$.
By definition of $b$ given in \eqref{eq:formetrib} and using standard norm estimates,  there exists a constant $C>0$ such that for any $\mu>0$:
\begin{equation*}
\begin{split}
|b(w^n,\psi^{n+1},\tilde{w}^{n+1})|&\leq \dfrac{1}{2 \mu} \Vert w^n \Vert^2_{\LL^2} 
\left( C \Vert \nabla \psi^{n+1} \Vert^2_{\LL^{\infty}} + \Vert \psi^{n+1} \Vert^2_{\LL^{\infty}} \right) + \dfrac{\mu}{4} \Vert \nabla \tilde{w}^{n+1} \Vert^2_{\LL^2}\\
& \leq C \Vert w^n \Vert^2_{\LL^2} + \dfrac{\mu}{4} \Vert \nabla \tilde{w}^{n+1} \Vert^2_{\LL^2}
\end{split}
\end{equation*}
We finally use the following equality:
\begin{equation}\label{eq:scalar}
(a-b,a) =\dfrac{1}{2} \left( \Vert a \Vert^2 -  \Vert b \Vert^2 + \Vert a-b \Vert^2  \right).
\end{equation}
The following estimate is obtained:
\begin{equation}\label{eq:stab1}
\begin{split}
&\dfrac{1}{2 \delta t} \left( \Vert \tilde{w}^{n+1} \Vert^2_{\LL^2} - \Vert w^n \Vert^2_{\LL^2}  + \Vert \tilde{w}^{n+1} - w^n \Vert^2_{\LL^2}\right)
+ \dfrac{\mu}{2} \Vert \nabla \tilde{w}^{n+1} \Vert^2_{\LL^2}\\
& + \dfrac{1}{\eta} \integre{\O}{}{\chi_{\o(t^{n+1})} | \tilde{w}^{n+1} |^2}{x}
+ (\nabla p^n , \tilde{w}^{n+1})_{\LL^2} \leq C \Vert F^{n+1} \Vert^2_{\LL^2} + C \Vert w^n \Vert^2_{\LL^2}.
\end{split}
\end{equation}
From \eqref{eq:vpp_lifting_correction} and \eqref{eq:vpp_lifting_gradpression}, we have:
\begin{equation}\label{eq:vpp_lifting_gradpression2}
\dfrac{w^{n+1}-\tilde{w}^{n+1}}{\delta t} + \nabla (p^{n+1}-p^n) = 0.
\end{equation}
By taking $w^{n+1}$ as a test function in \eqref{eq:vpp_lifting_gradpression2} and using  again \eqref{eq:scalar} we obtain:
\begin{equation}\label{eq:stab2}
\dfrac{1}{2 \delta t} \left( \Vert w^{n+1} \Vert^2_{\LL^2} - \Vert \tilde{w}^{n+1} \Vert^2_{\LL^2} + \Vert w^{n+1} - \tilde{w}^{n+1} \Vert^2_{\LL^2} \right)
+\left(\nabla p^{n+1} - \nabla p^n,w^{n+1}\right)_{\LL^2} = 0.
\end{equation}
Choosing $p^{n+1}$ as a test function in \eqref{eq:vpp_pression}, we have:
\begin{equation}\label{eq:stab3}
\dfrac{\e}{2} \left( \Vert p^{n+1} \Vert^2_{\LL^2}  - \Vert p^n \Vert^2_{\LL^2} + \Vert p^{n+1} - p^n \Vert^2_{\LL^2} \right) - \left( \nabla p^{n+1}, w^{n+1} \right)_{\LL^2} = 0.
\end{equation}
At last, taking $\nabla p^{n+1}$ as a test function in \eqref{eq:vpp_lifting_gradpression2} we obtain:
\begin{equation}\label{eq:stab4}
\dfrac{\delta t }{2} \left( \Vert \nabla p^{n+1} \Vert^2_{\LL^2} - \Vert \nabla p^n \Vert^2_{\LL^2} + \Vert \nabla p^{n+1} - \nabla p^n \Vert^2_{\LL^2} \right) + \left( \nabla p^{n+1}, w^{n+1}-\tilde{w}^{n+1}\right)_{\LL^2} = 0
\end{equation}
Finally, these four estimates \eqref{eq:stab1}, \eqref{eq:stab2}-\eqref{eq:stab4} are summed up. The sum of the scalar products reduces to
$\left( \nabla p^{n+1} - \nabla p^n, w^{n+1} - \tilde{w}^{n+1} \right)$, which is bounded using Young inequality:
$$
|\left(\nabla (p^{n+1}-p^n), w^{n+1}- \tilde{w}^{n+1}\right)_{\LL^2}| \leq
\dfrac{\delta t}{2}\Vert \nabla p^{n+1} - \nabla p^n \Vert^2_{\LL^2}+
\dfrac{1}{2 \delta t} \Vert w^{n+1} - \tilde{w}^{n+1} \Vert^2_{\LL^2}
$$
Therefore,
\begin{equation*}
\begin{split}
&\dfrac{1}{2 \delta t} \left( \Vert w^{n+1} \Vert^2_{\LL^2} - \Vert w^n \Vert^2_{\LL^2} + \Vert \tilde{w}^{n+1} - w^n \Vert^2_{\LL^2} \right) + \dfrac{\mu}{2} \Vert \nabla \tilde{w}^{n+1} \Vert^2_{\LL^2} \\
&+ \dfrac{\e}{2} \left( \Vert p^{n+1} \Vert^2_{\LL^2} - \Vert p^n \Vert^2_{\LL^2} + \Vert p^{n+1} - p^n \Vert^2_{\LL^2}\right) \\
& + \dfrac{\delta t}{2} \left( \Vert \nabla p^{n+1} \Vert^2_{\LL^2} - \Vert \nabla  p^n \Vert^2_{\LL^2} \right) + \dfrac{1}{\eta} \integre{\O}{}{\chi_{\o(t^{n+1})}|\tilde{w}^{n+1}|^2}{x} \leq C \Vert F^{n+1} \Vert^2_{\LL^2} + C \Vert w^n \Vert^2_{\LL^2}
\end{split}
\end{equation*}
This last equation is multiplied by $2 \delta t$ and written in $k$ instead of $n$. Finally, the equations is summed from $k=0$ to $n -1$ with $n \leq N = E\left( \dfrac{T}{\delta t} \right)$ where $E$ denotes the floor function, and we deduce:
\begin{equation}\label{eq:stability}
\begin{split}
&\Vert w^n \Vert^2_{\LL^2} + \delta t \e \Vert p^n \Vert^2_{\LL^2} + \delta t^2 \Vert \nabla p^n \Vert^2_{\LL^2} + \sum\limits_{k=0}^{n-1}{\Vert \tilde{w}^{k+1}-w^k \Vert^2_{\LL^2}} \\
& + \mu \sum\limits_{k=0}^{n-1}{\delta t \Vert \nabla \tilde{w}^{k+1} \Vert^2_{\LL^2}}
+  \e \sum\limits_{k=0}^{n-1}{\delta t\Vert p^{k+1} - p^k \Vert^2_{\LL^2}}
+\dfrac{2 }{\eta} \sum\limits_{k=0}^{n-1}{\delta t \integre{\O}{}{\chi_{\o(t^{k+1})} |\tilde{w}^{k+1}|^2}{x}} \\
& \leq \Vert w^0 \Vert^2_{\LL^2} + \e \delta t\Vert p^0 \Vert^2_{\LL^2} + \delta t^2\Vert \nabla p^0\Vert^2_{\LL^2} + 2C \sum\limits_{k=0}^{n-1}{\delta t \Vert F^{k+1} \Vert^2_{\LL^2}}\\
& + 2C \sum\limits_{k=0}^{n-1}{\delta t \Vert w^k \Vert^2_{\LL^2}}
\end{split}
\end{equation}
It implies:
$$
\Vert w^n \Vert^2_{\LL^2} \leq f_n + \sum\limits_{k=0}^{n-1}{g_k \Vert w^k \Vert^2_{\LL^2}}
$$
with,
\begin{equation*}
\left\lbrace
\begin{split}
& f_n = \Vert w^0 \Vert^2_{\LL^2} + \e \delta t \Vert p^0 \Vert^2_{\LL^2} + \delta t^2 \Vert \nabla p^0 \Vert^2_{\LL^2} + 2 C \sum\limits_{k=0}^{n-1}{\delta t \Vert F^{k+1}\Vert^2_{\LL^2}} \\
&g_k = 2C \delta t
\end{split}
\right.
\end{equation*}
The discrete Gronwall lemma \ref{lem:gronwall_discrete} thus gives the following upper bound on $\Vert w^n \Vert$:
$$
\Vert w^n \Vert^2_{\LL^2} \leq f_n (1 + 2 C T \exp(2 C T))
$$
Going back to \eqref{eq:stability} we deduce there exists $C>0$ such that:
\begin{equation}\label{eq:stability2}
\begin{split}
&\delta t \e \Vert p^n \Vert^2_{\LL^2} + \delta t^2 \Vert \nabla p^n \Vert^2_{\LL^2} + 
\sum\limits_{k=0}^{n-1}{\Vert \tilde{w}^{k+1} - w^k \Vert^2_{\LL^2}} \\
& + \mu \sum\limits_{k=0}^{n-1}{ \delta t \Vert \nabla \tilde{w}^{k+1} \Vert^2_{\LL^2}}
+  \e \sum\limits_{k=0}^{n-1}{\delta t \Vert p^{k+1} - p^k \Vert^2_{\LL^2}}
+ \dfrac{2}{\eta} \sum\limits_{k=0}^{n-1}{\delta t\integre{\O}{}{\chi_{\o(t^{k+1})} |\tilde{w}^{k+1}|^2}{x}} \leq C
\end{split}
\end{equation}
which concludes the proof of Proposition \ref{prop:stability2} and then of Proposition \ref{prop:stability} as, by assumptions, the quantities $f_n$ are uniformly bounded with respect to $n <N$ and to $\de t$.
\end{proof}
%
\begin{lmm}\label{lem:divergence}
Under the hypothesis of Proposition \ref{prop:stability}, we have
\begin{enumerate}[i)]
\item The divergence of $w_{\delta t}$ lies in $L^2(]0,T[;\LL^2)$ and there exists $C>0$ such that for any $\e >0$,
\begin{equation}\label{eq:estim_divergence}
\Vert \div \: w_{\delta t} \Vert_{L^2(]0,T[;\LL^2)} \leq C \sqrt{\e}
\end{equation}
As $\psi$ is divergence free, the same inequality holds for $v_{\delta t}$ which implies the strong convergence of $\div(v_{\delta t})$ towards $0$ when $\e$ tends to $0$.
\item $\hat{w}_{\delta t}$ is bounded in $L^2(]0,T[;\HH^{-1})$ and
$$
\Vert \hat{w}_{\delta t} \Vert^2_{L^2(]0,T[;\HH^{-1})} \leq C \delta t \dfrac{\delta t}{\e}
$$
\end{enumerate}
\end{lmm}
\begin{proof}[Proof of i)]
From the pressure equation \eqref{eq:vpp_pression} we have
$
\e (p^{n+1}-p^n) = -\div(w^{n+1})
$.
Then we have 
$$
\sum\limits_{k=0}^{N-1}{\delta t \Vert \div(w^{k+1}) \Vert^2_{\LL^2}} =
\e \sum\limits_{k=0}^{N-1}{\e \delta  t \Vert p^{k+1}-p^k \Vert^2_{\LL^2}}.
$$
 and  we deduce \eqref{eq:estim_divergence} exploiting the stability result \eqref{eq:stability2}.
\end{proof}
\begin{proof}[Proof of ii)]
The second point is proven using the correction equation \eqref{eq:vpp_lifting_correction}. Taking the $\HH^{-1}$-norm we obtain:
\begin{equation}\label{eq:estimwchap}
\Vert \hat{w}^{k+1} \Vert_{\HH^{-1}} \leq \dfrac{\delta t}{\e} \Vert \div(w^{k+1}) \Vert_{\LL^2}.
\end{equation}
Therefore, summing the square of this inequality from $k=0$ to $N-1$ and using the bound of the velocity's divergence \eqref{eq:estim_divergence}, we finally obtain:
$$
\sum\limits_{k=0}^{N-1}{\delta t \Vert \hat{w}^{k+1} \Vert^2_{\HH^{-1}}}
\leq \dfrac{C \delta t^2}{\e}.
$$
\end{proof}
\begin{lmm}\label{lem:translation}
Under the hypothesis of Proposition \ref{prop:stability}, we have
\begin{enumerate}[i)]
\item The velocity translation satisfies:
\begin{equation}\label{eq:translation}
\sum\limits_{k=0}^{N-1}{\Vert w^{k+1} - w^k \Vert^2_{\HH^{-1}}}
\leq 2 C \dfrac{\delta t}{\e} \left( \dfrac{\e}{\delta t} +1\right)
\end{equation}
\item $\nabla w_{\delta t}$ is bounded in $L^2(]0,T[;\LL^2)$ and
$$
\sum\limits_{k=0}^{N-1}{\delta t \Vert \nabla w^{k+1} \Vert^2_{\LL^2}} \leq
C \left( 1 + \dfrac{\delta t}{\e} \right)
$$
\end{enumerate}
\end{lmm}
\begin{proof}[Proof of i)]
The stability result \eqref{eq:stability2} gives a bound on the difference between the predicted velocity at the current time step and the velocity at the previous time step. Using the embedding of $\LL^2$ in $\HH^{-1}$ we deduce:
$$
\sum\limits_{k=0}^{N-1}{\Vert \tilde{w}^{k+1} - w^k \Vert^2_{\HH^{-1}}} \leq C.
$$
Then, combining  \eqref{eq:estimwchap} with Lemma \ref{lem:divergence} the following inequality holds:
\begin{equation*}
\begin{split}
\sum\limits_{k=0}^{N-1}{\Vert w^{k+1} - w^k \Vert^2_{\HH^{-1}}} &
\leq 2 \sum\limits_{k=0}^{N-1}{\left[\Vert w^{k+1} - \tilde{w}^{k+1} \Vert^2_{\HH^{-1}} + \Vert \tilde{w}^{k+1} - w^k \Vert^2_{\HH^{-1}}\right]} \\
& \leq 2 C \dfrac{\delta t}{\e} \left( 1+ \dfrac{\e}{\delta t}  \right)
\end{split}
\end{equation*}
\end{proof}
\begin{proof}[Proof of ii)]
To prove this point, we take $\hat{w}^{k+1}$ as a test function in the correction step \eqref{eq:vpp_lifting_correction}
\begin{equation*}
\begin{split}
\Vert \hat{w}^{k+1} \Vert^2_{\LL^2} + \dfrac{\delta t}{\e} \Vert \div(\hat{w}^{k+1}) \Vert^2_{\LL^2} & = - \dfrac{\delta t}{\e}  \left( \div(\tilde{w}^{k+1}), \div(\hat{w}^{k+1}) \right)_{\LL^2} \\
& \leq \dfrac{\delta t}{2 \e} \Vert \div(\hat{w}^{k+1}) \Vert^2_{\LL^2} + \dfrac{\delta t}{2 \e} \Vert \div(\tilde{w}^{k+1}) \Vert^2_{\LL^2}.
\end{split}
\end{equation*}
We thus obtain an estimate on the corrected velocity $\hat{w}^{k+1}$ and its divergence:
\begin{equation}\label{eq:stab_divergence}
\Vert \hat{w}^{k+1} \Vert^2_{\LL^2} + \dfrac{\delta t}{2 \e} \Vert \div(\hat{w}^{k+1}) \Vert^2_{\LL^2} \leq \dfrac{\delta t}{2 \e} \Vert \div(\tilde{w}^{k+1}) \Vert^2_{\LL^2}.
\end{equation}
Using that the norm $\Vert \cdot \Vert_{\HH^1}$ is equivalent to the norm $ \left( \Vert \cdot \Vert^2_{\LL^2} + \Vert \div( \cdot ) \Vert^2_{\LL^2} + \Vert \curl(\cdot)  \Vert^2_{\LL^2} \right)^{\frac 1 2} $ and $\curl(\hat{w}^{k+1})=0$, we obtain:

\begin{equation*}
\begin{split}
\Vert \hat{w}^{k+1} \Vert^2_{\HH^1} &\leq C \Big( \Vert \hat{w}^{k+1} \Vert^2_{\LL^2} + \Vert \div(\hat{w}^{k+1})  \Vert^2_{\LL^2} \Big) \\
& \leq \left( \dfrac{\delta t}{2 \e} + 1 \right) \Vert \nabla \tilde{w}^{k+1} \Vert^2_{\LL^2}
\end{split}
\end{equation*}
The previous inequality is summed up from $k=0$ to $N-1$. The predicted velocity gradient is bounded using the stability result \eqref{eq:stability2}. Then, we can  find an upper bound on the total velocity gradient:

\begin{equation*}
\begin{split}
\sum\limits_{k=0}^{N-1}{\Vert \nabla w^{k+1} \Vert^2_{\LL^2} \delta t} & \leq
2 \sum\limits_{k=0}^{N-1}{\delta t \Vert \nabla \tilde{w}^{k+1} \Vert^2_{\LL^2}}
+ 2 \sum\limits_{k=0}^{N-1}{\delta t \Vert \nabla \hat{w}^{k+1} \Vert^2_{\LL^2}}\\
& \leq  2 C + 2\left(1 + \dfrac{\delta t}{2 \e} \right) \sum\limits_{k=0}^{N-1}{ \delta t \Vert \nabla \tilde{w}^{k+1} \Vert^2_{\LL^2}} \\
& \leq 2 \left( 2 + \dfrac{\delta t}{2 \e} \right) C
\end{split}
\end{equation*}
\end{proof}
%
\section{Convergence analysis when $\e$ and $\delta t$ tend to 0}
\label{sec:convergence_epsilon}
A stability result has been obtained in the previous section. The main purpose of this section is to establish Theorem \ref{th:convergence_epsilon} which can be write as the following convergence theorem, when $\e$ and $\delta t$ tend to $0$ with $\e = \lambda \delta t$. 
\begin{thrm}[Convergence when $\e$ and $\delta t$ tend to $0$]\label{th:convergence_epsilon2} ~ \\
We assume that the hypothesis of Theorem \ref{th:convergence_epsilon} are satisfied.
Then, up to a subsequence, $(v^{n},p^{n})_n$ solution of \eqref{eq:vpp_prediction}-\eqref{eq:vpp_initial} converges towards $(v,p)$ weak solution of the penalized Navier Stokes problem \eqref{eq:NavierStokes} when $\e$ and $\delta t$ tend to $0$ with $\e = \lambda \delta t$.
Furthermore, $v$ and $p$ satisfy:
$$
v \in \Linf(]0,T[;\LL^2)\cap \Ldeux(]0,T[;\HH^1),
\qquad
p \in W^{-1,\infty}(]0,T[;\LL_0^2).
$$
Moreover, this solution is unique in two dimensional space.
\end{thrm}
\subsection{Weak convergence of the velocity}
We first establish the following result:
\begin{lmm}\label{lem:weak_convergence}
Under the hypothesis of Proposition \ref{prop:stability},
there exists $v \in L^2(]0,T[;\LL^2)$ (respectively $\tilde{v}\in L^2(]0,T[;\LL^2)$) such that, up to a subsequence, $v_{\delta t}$ (resp. $\tilde{v}_{\delta t}$) weakly converges towards $v$ (resp. $\tilde{v}$) when $\e$ and $\delta t$ tend to $0$ with $\e = \lambda \delta t$ :
\begin{enumerate}[i)]
\item $\left(v_{\delta t}\right)_{\delta t} \cvfaible v \text{ weakly in } L^2(]0,T[;\HH^1)$.
\item $\left( \tilde{v}_{\delta t} \right)_{\delta t} \cvfaible \tilde{v} \text{ weakly in }L^2(]0,T[;\HH^1)$
\end{enumerate}
Moreover, at the limit $\tilde{v} = v$
\end{lmm}
\begin{proof}
This result directly comes from the stability study. Indeed, the regularity of $\psi$ \eqref{eq:psi}, Proposition \ref{prop:stability} and lemma \ref{lem:translation} ensures that $\left(v_{\delta t}\right)_{\delta t}$ is uniformly bounded in $L^2(]0,T[;\HH^1)$ provided that $\e = \lambda \delta t$. Therefore, we can extract a subsequence still denoted $v_{\delta  t}$ that weakly converges towards a function $v$ in $ L^2(]0,T[;\HH^1)$.\\
In the same way, the sequence $\left( \tilde{v}_{\delta t}\right)_{\delta t}$ is bounded in $L^2(]0,T[;\HH^1)$. Then we can extract a subsequence that weakly converges towards $\tilde{v}$ in $L^2(]0,T[;\HH^1)$ when $\delta t$ and $\e$ tend to $0$ with $\e = \lambda \delta t$.\\

Let us show that $\tilde{v} = v$.
From lemma \ref{lem:divergence} when $\e = \lambda \delta t$
\begin{equation*}
\hat{w}_{\delta t} := w_{\delta t} - \tilde{w}_{\delta t} \underset{\delta t \rightarrow 0}{\cvfort} 0 \text{ strongly in }L^2(]0,T[;\HH^{-1}).
\end{equation*}
Moreover, Lemma \ref{lem:translation} and Proposition \ref{prop:stability} ensures that $w_{\delta t} - \tilde{w}_{\delta t}$ is bounded in $L^2(]0,T[;\HH^{1})$. Therefore,
\begin{equation}
w_{\delta t} - \tilde{w}_{\delta t} \cvfaible 0 \text{ weakly in } L^2(]0,T[;\HH^1)
\end{equation}
By unicity of the limit, we conclude that $\tilde{w} = w$ and $\tilde{v} = v$.
\end{proof}
\subsection{Strong convergence of the velocity}
In order to use Simon's results (Theorem \ref{th:simon}), let us estimate $w_{\delta t}$ in an appropriate Nikolskii space.
The following lemma ensures that $w_{\delta t}$ belongs to the Nikolskii space $N_2^{\frac 1 2}(]0,T[;\HH^{-1})$
\begin{lmm}\label{lem:kolmogorov}
Let $C_M$ be a positive constant and $h>0$.
Let $u$ defined on a time interval $[0,T)$ with values in $\HH^{-1}$. We denote $u_k$ the value of  $u$ at the time $t^k$ and $u_{\delta t}$ the step function defined as in \eqref{eq:defconstantparmaille}.\\
Moreover we assume that the following uniform upper bounds hold:
\begin{equation*}
\begin{split}
&\sum\limits_{k=0}^{N-1}{\Vert u^{k+1} - u^k \Vert^2_{\HH^{-1}}} \leq C_M \\
& \sup\limits_{k \leq N} \Vert u^k \Vert^2_{\HH^{-1}} \leq C_M
\end{split}
\end{equation*}
Then, there exists $C>0$ independent of $\delta t$ such that:
\begin{equation}\label{eq:kolmogorov01}
\integre{0}{T-h}{\Vert u_{\delta t}(t+h) - u_{\delta t}(t) \Vert_{\HH^{-1}}}{t} \leq C h^{\frac 1 2}
\end{equation}
And,
\begin{equation}\label{eq:kolmogorov02}
\left(\integre{0}{T-h}{\Vert u_{\delta t}(t+h) - u_{\delta t}(t) \Vert^2_{\HH^{-1}}}{t} \right)^{\frac 1 2} \leq C h^{\frac 1 2}
\end{equation}
\end{lmm}

We postpone the proof of this lemma in appendix \ref{sec:appendix}.
\begin{lmm}\label{lem:strongconvergence}
Let $p \in [2,+\infty[$. If $\delta t$ and $\e$ tend to $0$ with $\e = \lambda \delta t$ then $(v_{\delta t})_{\de t}$ strongly converges towards $v$ in $L^p(]0,T[;\LL^2)$.
\end{lmm}
\begin{proof}
To prove this lemma, we first show that $w_{\delta t}$ is uniformly bounded in the Nikolskii space $N_2^{\frac 1 2}(]0,T[;\HH^{-1})$. 
Combining \eqref{eq:translation} with Proposition \ref{prop:stability2} $i)$, from Lemma \ref{lem:kolmogorov} we deduce the following bound on the translations:
$$
\Vert \tau_h w_{\delta t} - w_{\delta t} \Vert^2_{L^2(]0,T-h[; \HH^{-1})} \leq C h^{\frac 1 2}
$$
which demonstrate that $w_{\delta t}$ belongs to the Nikolskii space $N_2^{\frac 1 2}(]0,T[;\HH^{-1})$.
Moreover, $w_{\delta t}$ is uniformly bounded in $L^2(]0,T[;\HH^1)$ (see the proof of Lemma \ref{lem:weak_convergence}), then we can apply Simon's theorem (Theorem \ref{th:simon}) with $B_0 = \HH^{1}$, $B_1 = \LL^2$ and $B_2 = \HH^{-1}$. It gives the strong convergence of $w_{\delta t}$ in $L^2(]0,T[;\LL^2)$.\\
Furthermore, from Lemma \ref{lem:weak_convergence}, the sequence $(w_{\delta t})$ weakly converges towards $w$ in $L^2(]0,T[;\LL^2)$.\\
Consequently, up to a subsequence, for $\e = \lambda \delta t$,
\begin{equation}
w_{\delta t} \underset{\delta t \rightarrow 0}{\cvfort}w \text{ strongly in }L^2(]0,T[;\LL^2).
 \end{equation} 
 Moreover, $\psi_{\delta t} \underset{\delta t \rightarrow 0}{\cvfort} \psi$. Consequently $v_{\delta t}$ strongly converges towards $v$ in $L^2([0,T];\LL^2)$.\\
 We know that $w_{\delta t}$ and $\psi_{\delta t}$ lie in $L^{\infty}(]0,T[;\LL^2)$, therefore
 \begin{equation}
 v_{\delta t} \cvfaible v \text{ weakly-}\star\text{ in }L^{\infty}(]0,T[;\LL^2)
 \end{equation}
 We obtain the result for any $p \in [2, +\infty[$ using interpolation properties.
\end{proof}
\subsection{Weak convergence of the inertia terms}
We show the weak convergence of the inertia term $B(v_{\delta t}(t- \delta t), \tilde{v}_{\delta t}(\delta t))$. 
\begin{lmm}\label{lem:weakconvergence_inertie}
If $\delta t$ and $\e$ tend to $0$ with $\e = \lambda \delta t$ then $(B(v_{\delta t},\tilde{v}_{\delta t}))_{\de t}$ weakly converges towards $B(v,v)$ in $L^p(]0,T[;\LL^q)$, 
with
$(p,q)=(\frac 4 3,\frac 4 3) $ in two dimensions and $(p,q)=(\frac 4 3,\frac 6 5) $ in three dimensions.
\end{lmm}
We distinguish the cases of two and three dimensional spaces.
\subsubsection{The three dimensional case}
For $d=3$, from H\"older's Inequalities, we have:
\begin{equation*}
\begin{split}
\Vert (v_{\delta t} . \nabla) \tilde{v}_{\delta t} \Vert_{L^2(]0,T[;\LL^1)}&
\leq \Vert v_{\delta t} \Vert_{L^{\infty}(]0,T[;\LL^2)}
  \Vert \nabla \tilde{v}_{\delta t} \Vert_{L^2(]0,T[;\LL^2)}\\ 
\end{split}
\end{equation*}
and,
\begin{equation*}
\begin{split}
\Vert (v_{\delta t} . \nabla) \tilde{v}_{\delta t} \Vert_{L^1(]0,T[;\LL^{\frac 3 2})} &\leq
\Vert v_{\delta t} \Vert_{L^2(]0,T[;\LL^6)} \Vert \nabla \tilde{v}_{\delta t} \Vert_{L^2(]0,T[;\LL^2)}\\
\end{split}
\end{equation*}
Thanks to Proposition \ref{prop:stability2} and Sobolev embeddings in three dimensions, the r.h.s. in the above estimates is uniformly bounded.
Then using interpolation theorems, we deduce that $(v_{\delta t} . \nabla)\tilde{v}_{\delta t} $ is uniformly bounded in $L^{\frac 4 3}(]0,T[;\LL^{\frac 6 5})$.
Therefore, in this space, the sequence $((v_{\delta t}.\nabla)\tilde{v}_{\delta t})_{\delta t}$ weakly converges towards a function $g$ that remains to determine.\\
Combining the continuity of  the application $(w,v)\mapsto (w . \nabla)v$  from $L^2(]0,T[;\LL^2) \times L^2(]0,T[;\HH^1)$ in $L^1(]0,T[;\LL^1)$ with the convergences of Lemma \ref{lem:weak_convergence} and Lemma \ref{lem:strongconvergence}, we deduce the convergence of the inertia term in the space $L^1(]0,T[;\LL^1)$:
$$
(v_{\delta t} . \nabla)\tilde{v}_{\delta t} \cvfaible (v. \nabla)v \text{ in }L^1(]0,T[;\LL^1)
$$
The weak convergence in the smaller space $L^{\frac 4 3}(]0,T[;\LL^{\frac 6 5})$ yields $g = (v. \nabla)v$ and:
\begin{equation}\label{eq:weakB1}
(v_{\delta t} . \nabla)\tilde{v}_{\delta t} \cvfaible (v . \nabla)v \text{ weakly in }L^{\frac{4}{3}}(]0,T[;\LL^{\frac 6 5})
\end{equation}
We now study the convergence of $\div(v_{\delta t}) \tilde{v}_{\delta t}$.
First, we know that:
\begin{equation}
\Vert \tilde{v}_{\delta t} \div(v_{\delta t})  \Vert_{L^2(]0,T[;\LL^1)} \leq
\Vert \nabla v_{\delta t} \Vert_{L^2(]0,T[;\LL^2)} \Vert \tilde{v}_{\delta t} \Vert_{L^{\infty}(]0,T[;\LL^2)}.
\end{equation}
From the pressure equation, the following inequality holds:
\begin{equation}
\Vert \hat{w}^{n+1} \Vert^2_{\LL^2} \leq  2 \delta t^2 \left( \Vert \nabla p^n\Vert^2_{\LL^2} + \Vert \nabla p^{n+1}\Vert^2_{\LL^2}\right)
\end{equation}
Then, the stability result \eqref{eq:stability2} ensures for all $n \leq N $ that $\delta t^2 \Vert \nabla p^n \Vert^2 \leq C$ and we obtain a bound of $\hat{w}_{\delta t}$ in $\LL^{\infty}(]0,T[;\LL^2)$. Since $w_{\delta t}$ also lies in $L^{\infty}(]0,T[;\LL^2)$, we deduce that $\tilde{w}_{\delta t}$ is bounded in $L^{\infty}(]0,T[;\LL^2)$.\\
Furthermore, from Hölder's inequality, a second inequality holds:
$$
\Vert \tilde{v}_{\delta t} \div(v_{\delta t})  \Vert_{L^1(]0,T[;\LL^{\frac 3 2})}
\leq \Vert \nabla v_{\delta t} \Vert_{L^2(]0,T[;\LL^2)} \Vert \tilde{v}_{\delta t} \Vert_{L^2(]0,T[;\LL^6)}
$$
Finally, from interpolation theorems we obtain a bound of $\tilde{v}_{\delta t}\div(v_{\delta t}) $ in $L^{\frac 4 3}(]0,T[;\LL^{\frac 6 5})$ and consequently the weak convergence of $\div(v_{\delta t}) \tilde{v}_{\delta t}$ towards a function $g$ that remains to determine.\\
The operator $(u,v) \mapsto \frac{1 }{2} \div(u) v$ is continuous from $\HH^1 \times \LL^2$ to $\LL^1$ so using the previous convergences, $\div(v_{\delta t})\tilde{v}_{\delta t}$ converges towards $\div(v) v$ in the space $L^1(]0,T[;\LL^1)$. As the product is bounded in the smaller space $L^{\frac 4 3}(]0,T[;\LL^{\frac 6 5})$, the result below holds:
\begin{equation}\label{eq:weakB2}
\tilde{v}_{\delta t} \div(v_{\delta t}) \cvfaible v\div(v) \text{ weakly in }L^{\frac 4 3}(]0,T[;\LL^{\frac 6 5})
\end{equation}
Thus for $d=3$ Lemma \ref{lem:weakconvergence_inertie} follows from \eqref{eq:weakB1} and \eqref{eq:weakB2}.
\subsubsection{The two dimensional case}
For a two dimensional space, we can demonstrate the convergence in a higher regularity space.\\
Indeed, as $v_{\delta t} \in L^{\infty}(]0,T[;\LL^2) \cap L^2(]0,T[;\HH^1)$ then by interpolation $v_{\delta t}$ lies in $L^4(]0,T[;\HH^{\frac 1 2})$. Yet, $\HH^{\frac 1 2}$ is embedded into $\LL^4$, therefore $v_{\delta t} \in L^4(]0,T[;\LL^4)$.\\
Then $(v_{\delta t} . \nabla)\tilde{v}_{\delta t}$ is bounded in $L^{\frac 4 3}(]0,T[;\LL^{\frac 4 3})$. Using the same arguments as above we deduce
$$
(v_{\delta t} . \nabla)\tilde{v}_{\delta t} \cvfaible (v. \nabla)v \text{ weakly in }L^{\frac 4 3}(]0,T[;\LL^{\frac 4 3})
$$
and,
$$
\tilde{v}_{\delta t} \div(v_{\delta t})  \cvfaible v \div(v) \text{ weakly in }L^{\frac 4 3}(]0,T[;\LL^{\frac 4 3}).
$$
It concludes the proof of Lemma \ref{lem:weakconvergence_inertie} in the two dimensional case.

\subsection{Proof of Theorem \ref{th:convergence_epsilon2}}
We can now pass to the limit in the numerical scheme.\\
Let $\phi \in \mathbf{V}$ (for the notations, see Section \ref{sec:notations}). The numerical scheme \eqref{eq:vpp_prediction}-\eqref{eq:vpp_initial} reads in variational formulation:
\begin{equation*}
\begin{split}
\dfrac{\mathrm{d}}{\mathrm{d}t}(v_{\delta t}(t),\phi) &+ 2 \mu (D(\tilde{v}_{\delta t}),D(\phi)) + (B(v_{\delta t}(t-\delta t),\tilde{v}_{\delta t}(t)),\phi)\\
& + \dfrac{1}{\eta} (\chi_{\o(t)}(\tilde{v}_{\delta t}(t)-v_s(t)),\phi)
= (f(t),\phi)
\end{split}
\end{equation*}
We multiply by a function $\theta \in \mathcal{C}^{1}(0,T)$ such that $\theta(T)=0$ and we integrate from $0$ to $T$. We do not have any information on the time derivative of the velocity. Therefore, the temporal term is integrated by part so that the time derivative holds on $\theta$:
\begin{equation}
\begin{split}
&- \integre{0}{T}{(v_{\delta t}(t),\phi)\,\theta'(t)}{t} - (v_{\delta t}(0),\phi)\,\theta(0)
+ 2 \mu \integre{0}{T}{(D(\tilde{v}_{\delta t}(t)),D(\phi))\:\theta(t)}{t}\\
& +\integre{0}{T}{(B(v_{\delta t}(t-\delta t),\tilde{v}_{\delta t}(t)),\phi) \: \theta(t)}{t}
+\dfrac{1}{\eta}\integre{0}{T}{(\chi_{\o(t)}(\tilde{v}_{\delta t}(t)-v_s(t)),\phi) \: \theta(t)}{t} \\
&= \integre{0}{T}{(f_{\delta t},\phi) \: \theta(t)}{t}
\end{split}
\end{equation}
We pass to the limit $\delta t \rightarrow 0$ in this last equation with $\e= \lambda \delta t$ using
Lemma \ref{lem:weak_convergence} and Lemma \ref{lem:weakconvergence_inertie}. It gives: 
\begin{equation}
\begin{split}
&-\integre{0}{T}{(v(t),\phi)\,\theta'(t)}{t} - (v(0),\phi)\,\theta(0)
+ 2 \mu  \integre{0}{T}{(D(v(t)),D(\phi)) \: \theta(t)}{t}\\
& + \integre{0}{T}{((v(t) . \nabla)v(t),\phi) \: \theta(t)}{t}
+\dfrac{1}{\eta} \integre{0}{T}{(\chi_{\o(t)}(v(t)-v_s(t)),\phi) \: \theta(t)}{t}\\
&= \integre{0}{T}{(f(t),\phi) \: \theta(t)}{t}
\end{split}
\end{equation}
Since $\e$ also tends to $0$ then from Lemma \ref{lem:divergence}, we have at the limit
$$
\div(v) = 0 \text{ on }\O.
$$
Applying the above equality for $\theta \in \mathcal{D}(0,T)$, we deduce the following equality in $\mathbf{V'}$:
\begin{equation}\label{eq:dansvprime}
\begin{split}
-\integre{0}{T}{v(t)\: \theta'(t)}{t} =&
\integre{0}{T}{\div(2 \mu D(v(t)))\:\theta(t)}{t} - \dfrac{1}{\eta}\integre{0}{T}{\chi_{\o(t)}(v(t)-v_s) \: \theta(t)}{t}\\
&-\integre{0}{T}{(v(t) . \nabla) v(t)\: \theta(t)}{t} + \integre{0}{T}{f(t)\:\theta(t)}{t}
\end{split}
\end{equation}
The operator $L \: : \: u \mapsto \div(2 \mu D(u))$ (respectively $B \: : \: (u,v) \mapsto B(u,v))$ is continuous from $\mathbf{V}$ to $\mathbf{V'}$ (resp. from $\mathbf{V} \times \mathbf{V}$ to $\mathbf{V'}$). Therefore, there exists $C>0$ such that:
\begin{equation}\label{eq:continuity_operator}
\begin{split}
\integre{0}{T}{\Vert \div( 2 \mu D(v)) \Vert_{\mathbf{V'}}}{t} &\leq
C \integre{0}{T}{\Vert v \Vert_{\mathbf{V}}}{t}\leq C \sqrt{T} \Vert v \Vert_{L^2(]0,T[;\mathbf{V})}.\\
\integre{0}{T}{\Vert (v.\nabla)v \Vert_{\mathbf{V'}}}{t} & \leq C \integre{0}{T}{\Vert v \Vert^2_{\mathbf{V}}}{t}\leq C \Vert v \Vert^2_{L^2(]0,T[;\mathbf{V})}.\\
\end{split}
\end{equation}
As \eqref{eq:dansvprime} is valid for any $\theta \in \mathcal{D}(]0,T[)$ we deduce that $v$ has a weak derivative in time which lies in $L^1(]0,T[;\mathbf{V'})$ and for almost every $t \in ]0,T[$:
\begin{equation}\label{eq:inVprime}
\dfrac{\partial v}{ \partial t} - \div(2 \mu D(v)) + (v.\nabla)v + \dfrac{1}{\eta}\chi_{\o(t)}(v - v_s) = f \text{ in }\mathbf{V'}.
\end{equation}
We now need to recover the initial data. Since $\dfrac{\partial v}{\partial t}$ belongs to $L^1(]0,T[;\mathbf{V}')$ and $v$ belongs to $L^2(]0,T[;\mathbf{V})$, we show using Proposition \ref{prp:initialdata} that $v$ is continuous with values in $\mathbf{V}'$ for the strong topology. Furthemore, by hypothesis $v(0) = v_0$ in the weak continuity sense with values in $\mathbf{V}'$. Therefore, the initial condition $v(0) = v_0$ is verified in the strong sense because the weak limit is unique.\\

From De Rham theorem, we can now deduce the existence of the pressure.
Let $G(t)$ be defined by:
\begin{equation*}
G(t) = - \div(2 \mu D(v)) + (v .\nabla)v + \dfrac{1}{\eta}\chi_{\o(t)}(v-v_s)-f
\end{equation*}
Thanks to \eqref{eq:inVprime}, for almost every $t \in ]0,T[$,
$$
\left< \dfrac{\mathrm{d}v}{\mathrm{d}t},\phi \right>_{\mathbf{V'},\mathbf{V}} + <G(t),\phi>_{\HH^{-1},\HH_0^1}=0.
$$
We integrate this last equation from $0$ to $t$. It gives:
$$
<v(t),\phi>_{\LL^2} - <v(0),\phi>_{\LL^2} + \left< \integre{0}{t}{G(\tau)}{\tau},\phi\right>
$$
It can be written under the form:
$$
<K(t),\phi>_{\HH^{-1},\HH_0^1} = 0,
$$
where
 $$K(t) = v(t) - v(0) - \integre{0}{t}{\div(2\mu D(v))}{\tau} + \integre{0}{t}{(v . \nabla)v}{\tau} + \dfrac{1}{\eta}\integre{0}{t}{\chi_{\o(\tau)}(v-v_s)}{\tau} - \integre{0}{t}{f}{\tau}.$$
 
Note that $K$ is weakly continuous in time with values in $\HH^{-1}$. Therefore, for all $t \in ]0,T[$ we deduce from De Rham theorem the existence of $\pi(t) \in \LL_0^2$ such that:
$$
K(t) = - \nabla \pi(t)
$$
Following the work of \cite{BO13} (chapter V), we show that $t \mapsto \pi(t)$ is weakly continuous in time with values in $\LL^2$. In particular, $\pi$ lies in the space $L^{\infty}(]0,T[;\LL_0^2)$. Indeed, if $g \in \LL^2$, there exists $h \in \HH_0^1$ such that $\div(h) = g - m(g)$ where $m(g)$ denotes the mean value of $g$ on $\O$. Then,
\begin{equation*}
\begin{split}
(\pi(t),g)_{\LL^2} &= (\pi(t),g-m(g))_{\LL^2} \text{ because }m(\pi) = 0\\
&=(\pi(t),\div(h))_{\LL^2}\\
&=-(\nabla \pi(t),h)_{\HH^{-1},\HH_0^1}\\
& = (K(t),h)_{\HH^{-1},\HH_0^1}
\end{split}
\end{equation*}
This quantity is continuous because $K$ is weakly continuous in time with values in $\HH^{-1}$.
We can then introduce the distribution $p = \dfrac{\partial \pi}{\partial t}$ which lies in the space $W^{-1,\infty}(]0,T[;\LL_0^2)$.
Taking test functions under the form $\dfrac{\partial \phi}{\partial t}$ with $\phi \in \mathcal{D}(]0,T[ \times \O)$, we show that the equation
\begin{equation}\label{eq:weaklimitconvergenceepsilon}
\dfrac{\partial v}{\partial t} - \div(2 \mu D(v)) + (v. \nabla)v + \dfrac{1}{\eta}\chi_{\o(t)}(v-v_s) + \nabla p = f
 \end{equation} 
 is satisfied in the sense of distributions.
 
In two dimensional space we can show the uniqueness of solutions of this equation using classical results (see \cite{BO13} chapter V).
%
\section{Convergence towards the Navier-Stokes Equations}
\label{sec:convergence_eta}
The aim of this section is to study the convergence when $\eta$ tends to $0$.
To do so, we consider the weak limit of the scheme when $\e$ and $\delta t$ tend to $0$ which verifies \eqref{eq:weaklimitconvergenceepsilon} and indice the solution by $\eta$.
\begin{thrm}[Convergence when $\eta$ tends to $0$.]\label{th:convergence_eta2} ~ \\
When $\eta$ tends to $0$, the sequence $\left( v_{\eta} \right)_{\eta}$ weakly converges towards a limit $v$ which satisfies
\begin{equation*}
v_{|\partial \o(t)} = v_s(t)
\end{equation*}
Furthermore, $v_{|\O \setminus \o(t)}$ is a weak solution of the Navier Stokes equations on $\O \setminus \o(t)$.\\
Moreover, there exists $h \in \mathbf{W}'$ such that:
\begin{equation}
\dfrac{1}{\eta} \chi_{\o(t)}(v-v_s) \cvfaible h \text{ in } \mathbf{W}'
\end{equation}
\end{thrm}
We first prove the following lemma:
\begin{lmm}\label{lem:estim_bord}
For all $t \in ]0,T[$, we have
$$
\integre{0}{t}{\Vert v_{\eta}-v_s \Vert^2_{L^2(\partial \o(\tau))}}{\tau} \leq C \eta^{\frac 1 2}
$$
\end{lmm}

\begin{proof}
For all $t \in ]0,T[$,
$$
\Vert v_{\eta}-v_s \Vert^2_{L^2(\partial \o(t))} \leq
C \Vert v_{\eta} - v_s \Vert_{L^2(\o(t))} \Vert v_{\eta} - v_s \Vert_{H^1(\o(t))}
$$
We integrate this last inequality from $0$ to $t$ and obtain:
\begin{equation*}
\begin{split}
\integre{0}{t}{\Vert v_{\eta} - v_s \Vert^2_{L^2(\partial \o(\tau))}}{\tau} &\leq
C \integre{0}{t}{\Vert v_{\eta} - v_s \Vert_{L^2(\o(\tau))} \Vert v_{\eta} - v_s \Vert_{H^1(\o(\tau))}}{\tau} \\
& \leq C \left( \integre{0}{t}{\Vert v_{\eta}-v_s \Vert^2_{L^2(\o(\tau))}}{\tau} \right)^{\frac 1 2}
\left( \integre{0}{t}{\Vert v_{\eta} - v_s \Vert^2_{H^1(\o(\tau))}}{\tau} \right)^{\frac 1 2}\\
& \leq C \left( \integre{0}{t}{\Vert v_{\eta}-v_s \Vert^2_{L^2(\o(\tau))}}{\tau} \right)^{\frac 1 2} 
    \Vert w_{\eta} \Vert_{L^2(]0,T[;\HH^1)}
\end{split}
\end{equation*}
From Lemma \ref{lem:translation}, $w_{\delta t}$ is bounded in $\HH^1$. Moreover, this bound is uniform in $\dfrac{\delta t}{\e}$. Using the lower semicontinuity of the norm for the weak topology, we obtain that $w_{\eta}$ is also bounded in $\HH^1$.
Finally, from the energy estimates of Section \ref{sec:stability}, we obtain:
\begin{equation*}
\integre{0}{t}{\Vert v_{\eta} - v_s \Vert_{L^2(\partial \o(\tau))}}{\tau} \leq C \eta^{\frac 1 2}
\end{equation*}
\end{proof}
Therefore, when $\eta \rightarrow 0$, the velocity on the immersed boundary $\partial \o(t)$ tends towards the obstacle velocity $v = v_s$ in the space $L^2(]0,T[;L^2(\partial \o(t)))$. \\
\begin{proof}[Proof of theorem \ref{th:convergence_eta2}] ~ 
From \eqref{eq:weaklimitconvergenceepsilon} we have:
$$
\dfrac{\partial v_{\eta}}{\partial t} - \div(2 \mu D(v_{\eta})) + (v_{\eta} . \nabla)v_{\eta} + \nabla p_{\eta} + \dfrac{1}{\eta} \chi_{\o(t)}(v_{\eta}-v_s) = f
$$
Let $\phi \in \mathbf{W}$. We have:
\begin{equation}\label{eq:conveta1}
\begin{split}
\integre{0}{T}{\left( \dfrac{1}{\eta} \chi_{\o(t)}(v_{\eta}-v_s) , \phi \right)}{t}
&= \integre{0}{T}{\left< v_{\eta} , \dfrac{\partial \phi}{\partial t} \right>_{\mathbf{V},\mathbf{V}'}}{t} 
 + (v_{\eta}(0), \phi(0))\\
&- \integre{0}{T}{2 \mu (D(v_{\eta}),D(\phi))+ ((v_{\eta} . \nabla)v_{\eta} + \nabla p_{\eta} + f,\phi)}{t}
\end{split}
\end{equation}
We use the continuity of the divergence and inertia operators as in \eqref{eq:continuity_operator} and that the estimates obtained through the stability study are uniform in $\eta$. The following inequality holds:
$$
\integre{0}{T}{\left| \dfrac{1}{\eta} (\chi_{\o(t)}(v_{\eta}-v_s),\phi) \right|}{t} \leq C \Vert \phi \Vert_{\mathbf{W}}
$$
This is true for all test functions $\phi \in \mathbf{W}$, we deduce that
$$
\dfrac{1}{\eta} \chi_{\o(t)}(v_{\eta}-v_s) \cvfaible h \text{ weakly in }\mathbf{W'}
$$
We can now pass to the limit in \eqref{eq:conveta1} using the continuity of the inertia and diffusion operators as well as the lower semi-continuity of the norm. It gives:
\begin{equation}
 \integre{0}{T}{\left< \dfrac{\partial v}{\partial t} , \phi \right>}{t}
+ \integre{0}{T}{<- \div(2 \mu D(v)) + (v . \nabla)v + \nabla p - f,\phi> + <h,\phi>}{t} =0
\end{equation}
As $h$ is the weak limit of $\dfrac{1}{\eta}\chi_{\o(t)}(v_{\eta}-v_s)$, then for every function $\phi$ such that $\text{supp}(\phi(t,\cdot)) \subset \O\setminus\o(t)$ for all $t$:
$$<h,\phi> = 0$$
At the limit when $\eta$ tends to $0$, the velocity on the immersed boundary tends to the solid velocity (see Lemma \ref{lem:estim_bord}). We find back a Dirichlet boundary condition on the obstacle boundary.
\end{proof}
\newpage
\appendix
\section{Kolmogorov lemma}\label{sec:appendix}

\begin{lmm}\label{lem:kolmogorov1}
Let $C_M$ a positive constant and $h>0$.\\
Let $u$ defined on a time interval $[0,T)$ with values in $\HH^{-1}$. We denote $u_k$ the value of  $u$ at the time $t^k$ and $u_{\delta t}$ the step function defined as in \eqref{eq:defconstantparmaille}.\\
Moreover we assume that the following conditions are verified:
\begin{equation*}
\begin{split}
&\sum\limits_{k=0}^{N-1}{\Vert u^{k+1} - u^k \Vert^2_{\HH^{-1}}} \leq C_M \\
& \sup\limits_{k \leq N} \Vert u^k \Vert^2_{\HH^{-1}} \leq C_M
\end{split}
\end{equation*}
Then, there exists $C>0$ independent of $\delta t$ such that:
\begin{equation}\label{eq:kolmogorov01}
\integre{0}{T-h}{\Vert u_{\delta t}(t+h) - u_{\delta t}(t) \Vert_{\HH^{-1}}}{t} \leq C h^{\frac 1 2}
\end{equation}
And,
\begin{equation}\label{eq:kolmogorov02}
\left(\integre{0}{T-h}{\Vert u_{\delta t}(t+h) - u_{\delta t}(t) \Vert^2_{\HH^{-1}}}{t} \right)^{\frac 1 2} \leq C h^{\frac 1 2}
\end{equation}
\end{lmm}

\begin{proof}
We distinguish two cases, $h \leq \delta t$ and $h>\delta t$
\begin{enumerate}
\item $h \leq \delta t$\\
For $t \in [t^k,t^{k+1}[$ and $t+h < t_{k+1}$:
\begin{equation*}
u_{\delta t}(t+h) - u_{\delta t}(t) = 0
\end{equation*}
For $t \in [t^k,t^{k+1}[$ and $t+h \geq t^{k+1}$:
\begin{equation}
u_{\delta t}(t+h) - u_{\delta t}(t) = u^{k+1}-u^k
\end{equation}
Consequently, on each interval $[t^k,t^{k+1}[$ the function $u(t+h) - u(t)$ is non-null on an intervall of size $h$. We have:
\begin{equation}
\begin{split}
\integre{0}{T-h}{\Vert u_{\delta t}(t+h) - u_{\delta t}(t) \Vert_{\HH^{-1}}}{t} & \leq
\sum\limits_{k=0}^{N-1}{h \Vert u^{k+1} - u^k \Vert_{\HH^{-1}}} \\
& \leq \left( \sum\limits_{k=0}^{N-1}{h} \right)^{\frac 1 2}
         \left( \sum\limits_{k=0}^{N-1}{h \Vert u^{k+1} - u^k \Vert^2_{\HH^{-1}}}\right)^{\frac 1 2}\\
&\leq h^{\frac 1 2}(T C_M)^{\frac 1 2}
\end{split}
\end{equation}
And, 
\begin{equation}
\begin{split}
\integre{0}{T-h}{\Vert u_{\delta t}(t+h) - u_{\delta t}(t) \Vert^2_{\HH^{-1}}}{t} &\leq
\sum\limits_{k=0}^{N-1}{h \Vert u^{k+1} - u^k \Vert^2_{\HH^{-1}}}\\
& \leq h C_M
\end{split}
\end{equation}

\item $h >\delta t$:\\
We first consider the case where $h$ is a multiple of $\delta t$. Let $j = E \left( \dfrac{h}{\delta t}\right)$.\\
We denote by $\chi_{k}$ the characteristic function of the interval $[t^k,t^{k+1}[$. Then, on the interval $[0;T-h]$, the function $u(t+h) - u(t)$ can be expressed as:
\begin{equation*}
u_{\delta t}(t+h) - u_{\delta t}(t) = \left[
\begin{split}
& \sum\limits_{k=0}^{j-1}{u^{k+j} \chi_{k+j}(t) - u^k \chi_k(t)} \\
& +\sum\limits_{k=j}^{2j-1}{u^{k+j} \chi_{k+j}(t) - u^k \chi_k(t)}\\
& \:\:\:\:\:\:\vdots \\
& + \sum\limits_{k=N-2j}^{N-j-1}{u^{k+j} \chi_{k+j}(t) - u^k \chi_k(t)}
\end{split}
  \right]
\end{equation*}
In the last sum many terms vanish and only the first and last terms are remaining.
\begin{equation}
u_{\delta t}(t+h) - u_{\delta t} = \left[ - \sum\limits_{k=0}^{j-1}{u^k \chi_k(t)}
+ \sum\limits_{k=	N-j}^{N-1}{u^k \chi_k(t)} \right]
\end{equation}
Therefore, if we integrate $\Vert u_{\delta t}(t+h) - u_{\delta t}(t) \Vert_{\HH^{-1}}$ between $0$ and $T-h$ we obtain considering $\integre{0}{T-h}{\chi_k(t)}{t}  = \delta t$ for any $k$:
\begin{equation}
\begin{split}
\integre{0}{T-h}{\Vert u_{\delta t}(t+h) - u_{\delta t}(t) \Vert_{\HH^{-1}}}{t} &=
\sum\limits_{k=0}^{j-1}{\delta t \Vert u^k \Vert_{\HH^{-1}}} + 
\sum\limits_{k=N-j}^{N-1}{\delta t \Vert u^k \Vert_{\HH^{-1}}} \\
&\leq  2 \left( \sum\limits_{k=0}^{j-1}{\delta t} \right)\sup_{k \leq N} \Vert u^k \Vert_{\HH^{-1}} \\
&\leq 2 C h
\end{split}
\end{equation}
Finally, after some manipulations:
\begin{equation}
\integre{0}{T-h}{\Vert u_{\delta t}(t+h) - u_{\delta t}(t) \Vert_{\HH^{-1}}}{t}
\leq C h^{\frac 1 2}
\end{equation}
Moreover, we know that the characteristic functions have disjoint supports. Therefore:
\begin{equation}
\begin{split}
\integre{0}{T-h}{\Vert u_{\delta t}(t+h) - u_{\delta t}(t) \Vert^2_{\HH^{-1}}}{t} &
= \sum\limits_{k=0}^{j-1}{\delta t \Vert u^k \Vert^2_{\HH^{-1}}}
 + \sum\limits_{k=N-j}^{N-1}{\delta t \Vert u^k \Vert^2_{\HH^{-1}}} \\
 & \leq 2 \left( \sum\limits_{k=0}^{j-1}{\delta t} \right) \sup_{k \leq N} \Vert u^k \Vert^2_{\HH^{-1}}\\
 & \leq 2 Ch
\end{split}
\end{equation}
We now consider the case where $h$ is not a multiple of $\delta t$. We still denote $j$ as $E \left( \dfrac{h}{\delta t} \right)$ and we set $h_1 = (j+1) \delta t - h$.\\
The function $u_{\delta t}(t+h)$ then reads:
$$
u^j \chi_{[t^0;h_1]}(t) + \sum\limits_{k=0}^{n-j-3}{\chi_{[h_1 + k \delta t ; h_1 + (k+1) \delta t]}(t) u^{k+j+1}} + u^{N-1} \chi_{[h_1 + (N-j-2) \delta t; T - h]}(t)
$$
Then the difference $u(t+h) - u(t)$ can be expressed as:
\begin{equation}
u_{\delta t}(t+h) - u_{\delta t}(t) =
\left[
\begin{split}
 &u^j \chi_{[t^0;h_1]}(t) \\
 & + \sum\limits_{k=0}^{n-j-3}{\chi_{[h_1 + k \delta t; h_1 + (k+1) \delta t]}(t) u^{k+j+1}}\\
 & + u^{N-1} \chi_{[h_1 + (N-j-2)\delta t; T-h]}(t) \\
 & - \sum\limits_{k=0}^{N-j-1}{\chi_k u^k}
\end{split}
\right]
\end{equation}
If we reorganize the characteristic functions, the last equation reduces to:
\begin{equation}
u_{\delta t}(t+h) - u_{\delta t}(t) = 
\left[
\begin{split}
&- \sum\limits_{k=0}^{j-1}{\chi_k(t) u^k} \\
& - u^j \chi_{[t^j;h]}(t) \\
& + u^{N-j-1} \chi_{[T-2h; h_1 + (N-2j -1) \delta t]} \\
& + \sum\limits_{k= N-2j-3}^{N-j-3}{\chi_{[h_1 + k \delta t; h_1 + (k+1) \delta t]}(t) u^{k+j+1}} \\
& + u^{N-1} \chi_{[h_1 + (N-j-2) \delta t; T-h]}
\end{split}
\right]
\end{equation}
We finally integrate this result from $0$ to $T-h$:
\begin{equation}
\begin{split}
\integre{0}{T-h}{\Vert u_{\delta t}(t+h) - u_{\delta t}(t) \Vert_{\HH^{-1}}}{t}
& \leq \sum\limits_{k=0}^j{ \delta t \Vert u^k \Vert_{\HH^{-1}}}
+ \sum\limits_{k=N-j-1}^{N-1}{\delta t \Vert u^k \Vert_{\HH^{-1}}} \\
& \leq 2 j \delta t \sup_{k \leq N} \Vert u^k \Vert_{\HH^{-1}} \\
& \leq 2 h C
\end{split}
\end{equation}
Using the same arguments as above we show:
\begin{equation}
\integre{0}{T-h}{\Vert u_{\delta t}(t+h) - u_{\delta t}(t) \Vert^2_{\HH^{-1}}}{t} \leq Ch
\end{equation}
\end{enumerate}
\end{proof}

\bibliographystyle{plain}
\bibliography{bibli}
\end{document}